\documentclass{amsart}
\usepackage{latexsym,amssymb,enumerate, amsmath}
\usepackage{graphicx}
\usepackage{xcolor}

\newtheorem*{thmA}{Theorem A}

\newtheorem{theorem}{Theorem}[section]

\newtheorem{lemma}[theorem]{Lemma}
\newtheorem{proposition}[theorem]{Proposition}

\newtheorem*{ThmB}{Theorem B}
\newtheorem*{ThmC}{Theorem C}

\theoremstyle{definition}
\newtheorem{example}[theorem]{Example}
\newtheorem{rem}[theorem]{Remark}
\newtheorem{defn}[theorem]{Definition}
\newenvironment{enumeratei}{\begin{enumerate}[\upshape (a)]}
    {\end{enumerate}}

\def\frat#1{{\bf \Phi}(#1)}

\def\cent#1#2{{\bf C}_{#1}(#2)}
\def\hall#1#2{{\rm Hall}_#1(#2)}
\def\syl#1#2{{\rm Syl}_#1(#2)}
\def\nor{\trianglelefteq\,}
\def\norm#1#2{{\bf N}_{#1}(#2)}
\def\oh#1#2{{\bf O}_{#1}(#2)}

\def\zent#1{{\bf Z}(#1)}
\def\sbs{\subseteq}

\def\fit#1{{\bf F}(#1)}

\def\V#1{{\rm V}(#1)}

\def\o#1{\overline{#1}}

\begin{document}

\title[Class graphs with a cut vertex]{Groups whose prime graph on class sizes has a cut vertex}

\author[S. Dolfi et al.]{Silvio Dolfi}
\address{Silvio Dolfi, Dipartimento di Matematica e Informatica U. Dini,\newline
Universit\`a degli Studi di Firenze, viale Morgagni 67/a,
50134 Firenze, Italy.}
\email{dolfi@math.unifi.it}

\author[]{Emanuele Pacifici}
\address{Emanuele Pacifici, Dipartimento di Matematica F. Enriques,
\newline Universit\`a degli Studi di Milano, via Saldini 50,
20133 Milano, Italy.}
\email{emanuele.pacifici@unimi.it}

\author[]{Lucia Sanus}
\address{Lucia Sanus, Departament de Matem\`atiques, Facultat de
 Matem\`atiques, \newline
Universitat de Val\`encia,
46100 Burjassot, Val\`encia, Spain.}
\email{lucia.sanus@uv.es}

\author[]{V\'ictor Sotomayor}
\address{V\'ictor Sotomayor, Instituto Universitario de Matem\'atica Pura y Aplicada (IUMPA-UPV), \newline Universitat Polit\`ecnica de Val\`encia, Camino de Vera s/n, 46022 Valencia, Spain.}
\email{vicorso@doctor.upv.es}

\thanks{The research of the first and  second  author is partially supported by the Italian PRIN 2015TW9LSR\_006 ``Group Theory and Applications". The research of the third  author is supported by  the Spanish  Ministerio de Economia y Competitividad proyecto MTM2016-76196-P,  partly with  FEDER funds. The  fourth author acknowledges the support of both the grant ACIF/2016/170 from Generalitat Valenciana (Spain) and the prize Borses Ferran Sunyer i Balaguer 2019 from Institut d'Estudis Catalans (Spain).}

\keywords{Finite Groups; Conjugacy Classes; Prime Graph}

\subjclass[2010]{20E45}

\begin{abstract}
Let \(G\) be a finite group, and let \(\Delta(G)\) be the \emph{prime graph} built on the set of conjugacy class sizes of \(G\): this is the simple undirected graph whose vertices are the prime numbers dividing some conjugacy class size of \(G\), two vertices \(p\) and \(q\) being adjacent if and only if \(pq\) divides some conjugacy class size of \(G\). In the present paper, we classify the finite groups \(G\) for which \(\Delta(G)\) has a \emph{cut vertex}.
\end{abstract}

\maketitle

\section{Introduction}


Let \(\Delta\) be a graph with \(n\) connected components; denoting by \(V\) the vertex set of \(\Delta\), an element \(v\in V\) is called a \emph{cut vertex} of \(\Delta\) if the number of connected components of the subgraph induced by \(V\setminus\{v\}\) in \(\Delta\) (i.e., the graph obtained by removing the vertex $v$ and all edges incident to $v$ from $\Delta$) is larger than \(n\).
If $\Delta$ is connected and it has a cut-vertex, then $\Delta$ is said to be \emph{$1$-connected}. 

Now, given a finite group \(G\), we consider the graph $\Delta(G)$ defined as follows:
its vertex set \(\V G\) consists of the prime numbers dividing the size of some conjugacy class of $G$, and two vertices $p$ and $q$ are adjacent in $\Delta(G)$ if and only if there exists a conjugacy class of $G$ having size divisible by the product $p q$. A well-established research field in the theory of finite groups investigates the interplay between graph-theoretical properties of \(\Delta(G)\) and the structure of \(G\) itself (see, for instance, the items in the References), and the present paper is a contribution in this framework; more specifically, our aim here is to describe the finite groups \(G\) such that the graph \(\Delta(G)\) has a cut vertex.
Note that, under this assumption, \(\Delta(G)\) is in fact $1$-connected. This follows from Theorem~4 of \cite{D0}: if \(\Delta(G)\) is disconnected then (it has two connected components and) the connected components are complete subgraphs, so \(\Delta(G)\) cannot have any cut vertex in this case.

We will show that $\Delta(G)$ has at most two cut vertices, and we will provide a
complete characterization of the structure of the group $G$, as well as of the graph \(\Delta(G)\), in both the cases when \(\Delta(G)\) has either one or two cut vertices. In the following statements, given a graph $\Delta$ with vertex set $V$, for $r\in V$ we denote by $\Delta -r$ the subgraph induced by \(V\setminus\{r\}\) in \(\Delta\); recall also that the vertex \(r\) of \(\Delta\) is called \emph{complete} if it is adjacent to all the other vertices of~\(\Delta\). 

\begin{thmA}
  Let $G$ be a finite group such that \(\Delta(G)\) has a cut vertex \(r\). Then the following conclusions hold.
  \begin{enumeratei} 
  \item \(G\) is a solvable group whose Fitting height is at most \(3\).
  \item $\Delta(G) - r$ is a graph with two connected components, that are both complete graphs. 
  \item If \(r\) is a complete vertex of \(\Delta(G)\), then it is the unique complete vertex and the unique cut vertex of \(\Delta(G)\). If \(r\) is non-complete, then \(\Delta(G)\) is a graph of diameter \(3\), and it can have at most two cut vertices.
  \end{enumeratei}
\end{thmA}

In order to state Theorem~B, which provides a much deeper description of the finite groups \(G\) such that \(\Delta(G)\) is \(1\)-connected, we need to introduce some terminology. We say that a finite group \(G\) is \emph{reduced}  if it does not have any non-trivial normal (equivalently, central) subgroup \(Z\) with \(G'\cap Z=1\). As one quickly realizes (see Proposition~\ref{rid}), the set of conjugacy class sizes of $G$ is the same as the set of conjugacy class sizes of the factor group $G/Z$ by any such subgroup \(Z\); moreover, it is not difficult to see that if \(Z\) is maximal with respect to the above property, then \(G/Z\) is reduced. In view of these remarks, it is sensible and not restrictive to focus on reduced groups. 


\begin{ThmB}
  Let $G$ be a finite reduced group. Then the  graph $\Delta(G)$ has a cut vertex $r$ if and only if, denoting by \(\alpha\) and \(\beta\) the vertex sets of the two complete connected components of \(\Delta(G)-r\), we have
  $G = ABR$ where $A \in \hall{\alpha}G$, $B \in \hall{\beta}G$, $R \in \syl rG$ are all non-trivial, $AB$ is an $r$-complement of $G$, $A$ and $B$ are abelian, and (up to interchanging \(\alpha\) and \(\beta\)) one of the following holds.

\medskip
  \begin{enumeratei}
  \item[{\bf{(I)}}] The Fitting subgroup $\fit G$  is $R$, and the set $\beta$ consists of a single prime $q$. Also, \(\fit{AB}=A\) is cyclic, and $|B| = q$, so \(G\) is nilpotent by metacyclic, of Fitting height \(3\). Furthermore, 
    for all $x \in R$, either

   \begin{description}
\item[(i)]    $A^y \leq \cent Gx$ for some $y \in R$ or
\item[(ii)]    $B^g \leq \cent Gx$ for some $g \in G$ and $\cent Ax \leq \zent {AB}$.
   \end{description}
  
  \bigskip
  \item[{\bf{(II)}}] $\fit G = A \times R$ (so \(G\) is nilpotent by abelian, in fact metabelian if \(R\) is abelian), $Z = \zent {AB} < B$,  $AB/Z$  is a Frobenius group with kernel \(AZ/Z\), and
    either
    \begin{description}
    \item[(IIa)] $R$ is abelian, $\cent BR = 1$, $Z \neq 1$ and $\cent Bx \leq Z$ for every non-trivial $x \in R$; or

    \item[(IIb)] $R$ is non-abelian and either
      \begin{description}
      \item[(IIb(i))] $G = R \times AB$; or
      \item[(IIb(ii))] $\cent Bx \leq Z$ for all $x \in R$ such that $\cent Gx R < G$.
      \end{description}
    \end{description}
    
 \bigskip   
  \item[{\bf{(III)}}] 
   Up to replacing \(R\) by a \(G\)-conjugate of it, we have that \(BR\) is a nilpotent subgroup of \(G\); furthermore, $\fit G = A \times R_0$ with $R_0 < R$, $\cent AR = 1$, 
and $[A,B]BR/R_0$ is a Frobenius group with kernel \([A,B]R_0/R_0\). In particular, \(G\) is metanilpotent, in fact metabelian if \(R\) is abelian; in this case, we also have $R_0 = 1$ and $\cent AB \neq~1$.

  \end{enumeratei}
\end{ThmB}

In Section~5 we will discuss the various types of groups that appear in Theorem~B, and we will describe the structure of the relevant graphs. We will also see, in Example~\ref{ex}, that the graphs having a cut vertex, which can be realized as $\Delta(G)$ for some finite group $G$, are precisely the $1$-connected graphs whose vertices are covered by two complete subgraphs.

It turns out that the finite reduced groups $G$ for which $\Delta(G)$ has two cut vertices constitute a subclass of the groups described in {\bf(IIa)} of Theorem~B (with respect to one of the cut vertices, whereas they are a subclass of {\bf(III)} with respect to the other; see Remark~\ref{thmC}). In the following statement, we refer to the notation introduced in Theorem~B.

\begin{ThmC}
  Let $G$ be a finite reduced group. Then the  graph $\Delta(G)$ has two distinct cut vertices \(r\) and \(t\) if and only if, with respect to \(r\) (say), the following holds.
  \begin{enumeratei}
  \item \(G\) is as in {\bf(IIa)} of Theorem~{\rm B}, with \(t\) lying in \(\beta\).
  \item Denoting by \(B^*\) the Hall \(t'\)-subgroup of \(B\), we have \(B^*\leq Z\), and \(B\) acts fixed-point freely (by conjugation) on \([R,B^*]\).
\end{enumeratei}
\end{ThmC}

We briefly digress with the following remark. One problem that may be of interest, concerning the graph \(\Delta(G)\), is to \emph{understand the situation when \(\Delta(G)\) does not contain any cycle}. This property clearly holds if the graph has at most two vertices; moreover, as we have a complete control of the case when \(\Delta(G)\) is disconnected (via Theorem~4 of \cite{D0}), the relevant question in this context is to \emph{classify the finite reduced groups \(G\) such that \(\Delta(G)\) is acyclic, connected, with at least three vertices}. It turns out that this problem is strongly related to our present discussion. 

In fact, since the vertices of \(\Delta(G)\) can be partitioned in two subsets each inducing a complete subgraph (\cite[Corollary B]{DPSS}), it is easily seen that \(\Delta(G)\) has at most four vertices if it is acyclic. Therefore, the purpose is to describe the groups \(G\) for which \(\Delta(G)\) is a path of length two or three. In both cases \(\Delta(G)\) has a cut vertex (actually two of them in the latter case), therefore Theorem~B and Theorem~C enable us to complete this classification.

To close with, we mention that the study of cut vertices for the \emph{character degree graph} (i.e. the graph obtained by considering the degrees of irreducible characters, instead of the sizes of the conjugacy classes) has been carried out by M.L. Lewis and Q. Meng in~\cite{LM}. 

All the groups considered in the following discussion are tacitly assumed to be finite groups.




\section{Preliminary results}

For a positive integer \(n\), we define \(\pi(n)\) to be the set of prime divisors of \(n\); if \(G\) is a group, \(\pi(G)\) will stand for \(\pi(|G|)\).

Next, we gather some well-known facts concerning conjugacy class sizes of a group.
Given an element $x$ of the group $G$, denote by $x^G$ the conjugacy class of $x$ in $G$, and by \(\pi_G(x)\) the set of prime divisors of \(|x^G|\): if $N$ is a normal subgroup of  $G$ then, for any $x \in G$, we have
$\pi_{G/N}(xN) \subseteq \pi_G(x)$ and, for $y \in N$, we have $\pi_N(y) \subseteq \pi_G(y)$.  Another elementary remark is that a prime number $p$ does not belong to \(\V G\) if and only if $G$ has a central Sylow $p$-subgroup. Furthermore, the following holds.

\begin{proposition}
\label{rid}
Let \(Z\) be a normal subgroup of \(G\)   such that \(G'\cap Z=1\). Then $Z \leq \zent G$ and the set of conjugacy class sizes of \(G/Z\) is the same as the set of conjugacy class sizes of \(G\).
\end{proposition}

\begin{proof} As $[G, Z ] \leq G'\cap Z$, it is clear that $Z$ is contained in the center of $G$. It will be enough to show that, for every \(x\in G\), we have \(\cent{G/Z}{xZ}=\cent G x/Z\). In fact, if \(yZ\) lies in \(\cent{G/Z}{xZ}\), we get \([x,y]\leq G'\cap Z=1\), and therefore \(y\) lies in \(\cent G x\); this proves that \(\cent{G/Z}{xZ}\subseteq \cent G x/Z\), and equality clearly holds.
\end{proof} 

As mentioned in the Introduction, the group \(G\) is said to be \emph{reduced} if it does not have any non-trivial subgroup \(Z\) as in the hypothesis of the above proposition, and it not restrictive to focus on reduced groups for the purposes of this paper. Note that, for a reduced group \(G\), we have \(\V G=\pi(G)\).

In the following proposition, we recall the description of the groups $G$ such that $\Delta(G)$ is disconnected.

\begin{proposition}[\mbox{\cite[Theorem~4]{D0}}]
  \label{disconnected}
  Let $G$ be a group. Then 
  the graph $\Delta(G)$ is disconnected if and only if $G=AB$, where
  $A\trianglelefteq G$ and $B$ are abelian Hall subgroups of $G$ of coprime order, and
  $G/Z$, where $Z = \zent G$, is a Frobenius group with Frobenius kernel \(AZ/Z\).
  In this case $\Delta(G)$ has two connected components, with vertex sets \(\pi(AZ/Z)\) and \(\pi(BZ/Z)\) respectively, that are both complete. 
  
\end{proposition}

The next lemma is well known and easy to prove. After that, we recall some statements that will come into play, dealing with non-complete vertices of \(\Delta(G)\).
 
\begin{lemma} \label{product}
  Let $G$ be a group and let $x,y \in G$ be such that one of the following holds.
  \begin{enumeratei}
  \item $x$ and $y$ have coprime orders and they commute.
  \item $x \in X $ and $y \in Y$, where $X$ and $Y$ are normal subgroups of $G$ such that $X \cap Y = 1$. 
\end{enumeratei}
Then \(\pi_G(x) \cup \pi_G(y)  \sbs \pi_G(xy). \)
\end{lemma}

Given a prime \(p\), as customary, we say that a group is \(p\)-nilpotent if it has a normal Hall \(p'\)-subgroup.

\begin{proposition}
\label{nilpotency}
Let \(G\) be a group; then the following holds.
\begin{enumeratei}
\item Let \(p\), \(q\) be non-adjacent vertices of \(\Delta(G)\). Then \(G\) is either \(p\)-nilpotent or \(q\)-nilpotent, with both abelian Sylow \(p\)-subgroups and Sylow \(q\)-subgroups.
\item If $\pi$ is a set of vertices which are all  non-adjacent to a vertex $p$ in $\Delta(G)$,
  then $G$ is $\pi$-solvable with abelian Hall $\pi$-subgroups, and the vertices in \(\pi\) are pairwise adjacent. 
\end{enumeratei}
\end{proposition}
\begin{proof}
Part (a) comes from~ \cite[Lemma~5 and Theorem~B]{CD2} and part (b) from~\cite[Theorem~C]{CDPS13}.
  \end{proof}
  
We remark that the last conclusion in part (b) of Proposition~\ref{nilpotency} follows from a much more general fact, that will be crucial in our discussion, and that was already mentioned in the Introduction. This is Corollary~B in \cite{DPSS}:

\begin{theorem}
\label{ultimo}
Let \(G\) be a group. Then the vertex set of \(\Delta(G)\) can be partitioned into two subsets, each inducing a complete subgraph of \(\Delta(G)\).
\end{theorem}

\begin{lemma}\label{three}
  Let $p,r,q$ be three distinct primes and let $G = PRQ$, where $P \in \syl pG$, $R \in \syl rG$,
  $Q \in \syl qG$, $RQ \leq G$, and both $P$ and $PR$ are normal subgroups of $G$. If $\{ p, q\}$ is not an edge of
  $\Delta(G)$, then $R$ centralizes either $P$ or $Q$.  
  \end{lemma}

\begin{proof}
  Note that, as \(PR\nor G\), we have \(R=PR\cap RQ\nor RQ\). Also, we can assume that both $p$ and $q$ are vertices of $\Delta(G)$, as otherwise either $P$ or $Q$ are
  central in $G$. Now, Theorem 24 of~\cite{BDIP} yields that either $R \nor G$, and hence $[R, P] = 1$,
  or $PQ \nor G$. In the latter case, as above, we have \(Q=PQ\cap RQ\nor RQ\); therefore both $R$ and $Q$ are normal subgroups of $RQ$, and $[R,Q] = 1$.
  \end{proof}

The following lemma introduces an important characteristic subgroup of \(G\), that we denote by  \(K_p(G)\), associated to a non-complete vertex \(p\) of $\Delta(G)$. Before stating it, we introduce some more notation.

\begin{defn}
For a group \(G\), we denote by $\nu(G)$ the set of the primes $t\in\pi(G)$ such that
$G$ has a \emph{normal} Sylow $t$-subgroup. 
\end{defn}
\begin{lemma}{\cite[Lemma~2.3]{DPSS}.} \label{vertex}
  Let $G$ be a group,  let $p$ be a non-complete vertex of $\Delta(G)$ and
  $P$ a Sylow $p$-subgroup of $G$.
  Then $G$ is $p$-solvable, $P$ is abelian,  and $[G, P]$ has a normal
  $p$-complement $K_p(G)$.
 Furthermore,  $[K_p(G), P] = K_p(G)$ and, if $ p \not\in \nu(G)$, then there are elements $x$ in $K_p(G)$ such that
 $p \in \pi_G(x)$. 
\end{lemma}

We  note that, using the bar convention in a factor group $\overline {G} = G/N$ (for $N\nor G$),
we have $\overline{[G,P]} = [\overline{G}, \overline{P}]$, so the image of $K_p(G)$ along the canonical projection is the normal $p$-complement of  $[\overline{G}, \overline{P}]$. In particular, if \(p\) is a non-complete vertex also for \(\Delta(\o G)\), then $\overline{K_p(G)} = K_p(\overline{G})$ holds. 
We also observe that $p \in \nu(G)$ if and only if $K_p(G) = 1$.

Further, we need a  basic result related to the existence of regular orbits in coprime actions
of abelian groups.

\begin{lemma}{\cite[Lemma~2.4]{DPSS}.}\label{action}
  Let $G$ be a group such that $G/\fit G$ is abelian. Then there exists an element $g \in G$ such that the set of all prime divisors of $|G/\fit G|$ is contained in $\pi_G(g)$.
\end{lemma}

Finally, we are ready to state a key preliminary result. We refer to the notation introduced in Lemma~\ref{vertex}. 

\begin{proposition}\label{gamma}
  Let $G$ be a group. Assume that $p$ and $q$ are non-adjacent vertices of $\Delta(G)$, and denote by $P$ and \(Q\) a Sylow $p$-subgroup and a Sylow $q$-subgroup of $G$, respectively. Assume further that  $M = K_p(G)$ is a minimal normal subgroup of \(G\), and that $Q$ is not normal in $G$.
  Then $M$ is abelian, it has a complement in \(G\), and the following conclusions hold.
  \begin{enumeratei}
  \item \(\oh q G=Q\cap\cent G M\).
  \item
    $\o G = G/\cent GM$ is a \(q\)-nilpotent group, \(\fit{\o G}\) is a cyclic group acting  fixed-point freely and irreducibly on $M$, and \(\o G/\fit{\o G}\) is cyclic as well. Also, $1\neq \o{P} \leq \fit{\o G}$ and $\o{Q} \cap \fit{\o G} = 1$.
\item Setting \(|M|=r^m\), we have that $|\o Q|$ divides $m$; also, \(q\) does not divide \(r^m-1\), and
$(r^m -1)/(r^{m/|\o Q|} -1)$ divides $|\fit{\o{G}}|$.
\item If \(N\) is a normal subgroup of G such that \(N\cap M=1\), then 
\(Q\leq\cent G N\).
\item \(\oh p G=P\cap\zent G\).

\end{enumeratei}
\end{proposition}

\begin{proof} This is a reformulation of Proposition~3.1 in \cite{CDPS13} and Proposition~2.5 in \cite{DPSS}; the proof of \cite[Proposition~2.5]{DPSS} includes an explanation of the fact that the hypotheses of \cite[Proposition~3.1]{CDPS13} are fulfilled under our assumptions.
\end{proof}

We conclude this preliminary section with an application of the tools introduced so far.

\begin{proposition}
  \label{pieces}
  Let $G$ be a group, and let $\alpha$, $\beta$ be non-empty and disjoint vertex subsets of $\Delta(G)$ such that there are no edges of \(\Delta(G)\) having one extreme in \(\alpha\) and the other in \(\beta\).
  Assume also that $\nu(G) \cap \alpha= \emptyset = \nu(G) \cap \beta$.
Then, up to interchanging $\alpha$ and $\beta$, there exists a normal subgroup $K$ of $G$ such that
  $K=K_p(G)$ for all $p \in \alpha$ and $K < K_q(G)$ for all $q \in \beta$. 
\end{proposition}

\begin{proof}
 For $p \in \alpha$ and $q \in \beta$, consider the subgroups \(K_p=K_p(G)\) and \(K_q=K_q(G)\):
we will first show that, say, $K_p < K_q$.
Set \(N=K_p\cap K_q\) and assume, working by contradiction, that $N$ is  a proper subgroup of both
$K_p$ and $K_q$. In particular, \(p\) and \(q\) are both (non-complete) vertices of \(\Delta(G/N)\) as well, therefore, as remarked in the paragraph following Lemma~\ref{vertex}, we have $K_p/N = K_p(G/N)$ and $K_q/N = K_q(G/N)$. Now, an application of Lemma~\ref{vertex} to the factor group \(G/N\) yields that there exist two elements \(x\in K_{p}\), \(y\in K_{q}\) such that \(p\in\pi_{G/N}(xN)\) and \(q\in\pi_{G/N}(yN)\); by Lemma~\ref{product}(b), we see that \(pq\) divides \(|(xyN)^{G/N}|\), thus it divides \(|(xy)^G|\) contradicting the fact that \(p\) and \(q\) are non-adjacent in \(\Delta(G)\). We conclude that (say) \(K_{p}=N\), whence \(K_{p}\leq K_{q}\). Also, if \(L\) is a normal subgroup of \(G\) such that \(K_{p}/L\) is a chief factor of \(G\) (so, as above, $K_p/L = K_p(G/L)$), then we can apply Proposition~\ref{gamma}(b) to the group \(G/L\), obtaining that $\o G = G/\cent G{K_{p}/L}$ has a normal Sylow \(p\)-subgroup, and a (non-trivial) Sylow \(q\)-subgroup intersecting \(\fit{\o G}\) trivially. In particular, the roles of \(p\) and \(q\) are not symmetric, and therefore the inclusion of \(K_{p}\) in \(K_{q}\) must be proper. Up to interchanging $p$ and $q$, we thus have $K_p < K_q$.

Next, we claim that \(K_{p_0} < K_{q}\) for every choice of $p_0 \in \alpha$.
In fact, assuming this does not hold, the paragraph above yields \(K_{q} <  K_{p_0}\); working in the factor group $\o G = G/\cent G{K_{p}/L}$ as above,
by Lemma~\ref{action} we have that $p_0$ does not divide $|\o G /\fit{\o G}|$, so
$\o{K_{p_0}} = 1$, a contradiction as $\o{K_{p_0}} \geq \o{K_q} > 1$. 
 Note that, by essentially the same argument, we can see that \(K_{p} < K_{q_0}\) holds as well for every choice of \(q_0 \in \beta \). 

We work now to show that, for every choice of $p , p_0 \in \alpha$, we have \(K_{p} = K_{p_0}\).  
First, let us see that one of these two subgroups is contained in the other. For a proof by contradiction, assume that \(N=K_{p}\cap K_{p_0}\) is properly contained in both \(K_{p}\) and \(K_{p_0}\). So, we can take normal subgroups \(L\) and \(L_0\) of \(G\), containing \(N\), such that \(K_{p}/L\) and \(K_{p_0}/L_0\) are chief factors of \(G\). Let \(Q\) be a Sylow \(q\)-subgroup of \(G\), where \(q\) lies in \(\beta\); by Proposition~\ref{gamma}(d) applied to the factor group \(G/L\), the normal subgroup \(K_{p_0}L/L\) (which intersects \(K_{p}/L\) trivially) is centralized by \(QL/L\), therefore \([K_{p_0}, Q]\leq L\). But clearly \([K_{p_0}, Q]\) also lies in \(K_{p_0}\), hence it lies in \(N\). In particular, \(QL_0/L_0\) centralizes \(K_{p_0}/L_0\), and thus  Proposition~\ref{gamma}(a) (applied to $G/L_0$) yields $QL_0/L_0 \nor G/L_0$, so $K_q \leq L_0 \leq K_p$, a contradiction by the previous paragraph. Our conclusion so far is that (say) \(K_{p}\leq K_{p_0}\), and it remains to show that equality holds. To this end, setting \(\o G=G/\cent{G}{K_{p}/L}\), observe first that \(\o{K_{p_0}}=1\). Otherwise, setting \(P_0\) to be a Sylow \(p_0\)-subgroup of \(G\), \(\o {K_{p_0}}\) would be a non-trivial (normal) \(p_0'\)-subgroup of \([\o G, \o{P_0}]\), thus \(\o G\) would not have a normal Sylow \(p_0\)-subgroup, yielding \(p_0\mid |\o G/\fit {\o G}|\); but Proposition~\ref{gamma}(b) ensures that also \(q\) divides \(|\o G/\fit {\o G}|\), so that (by Lemma~\ref{action}) \(p_0q\) divides the size of some conjugacy class of \(\o G\), a contradiction. Finally, we know by Proposition~\ref{gamma} that \(K_{p}/L\) has a complement \(H/L\) in \(G/L\), so, in particular, \(K_{p_0}=K_{p}(K_{p_0}\cap H)\); as \((K_{p_0}\cap H)/L\) is normal in \(H/L\) and it centralizes \(K_{p}/L\), we get that \((K_{p_0}\cap H)/L\) is a normal subgroup of \(G/L\) intersecting \(K_{p}/L\) trivially. An application of Proposition~\ref{gamma}(d) to the factor group \(G/L\) gives \([K_{p_0}\cap H, Q]\leq L\), whence \([K_{p_0},Q]=[K_{p}(K_{p_0}\cap H), Q]\leq K_{p}\). But now, if \(K_{p_0}\) is strictly larger than \(K_{p}\), we can take a subgroup \(L_0\) of \(G\), containing \(K_{p}\), such that \(K_{p_0}/L_0\) is a chief factor of \(G\). Proposition~\ref{gamma}(a) applied to \(G/L_0\) yields \([K_{p_0},Q]\not\leq L_0\), a contradiction. We conclude that, in fact, \(K_{p_0}=K_{p}\) holds. 

Therefore, we have proved that $K = K_ p < K_q$ for all $p \in \alpha$ and $q \in \beta$.
\end{proof}

\section{Proof of Theorem A}

In this section we prove Theorem~A, whose statement is recalled next. Our proof relies essentially on Theorem~\ref{ultimo}, and on some easy graph-theoretical considerations.

\begin{thmA}
  Let $G$ be a group such that \(\Delta(G)\) has a cut vertex \(r\). Then the following conclusions hold.
  \begin{enumeratei} 
  \item \(G\) is a solvable group whose Fitting height is at most \(3\).
  \item $\Delta(G) - r$ is a graph with two connected components, that are both complete subgraphs. 
  \item If \(r\) is a complete vertex of \(\Delta(G)\), then it is the unique complete vertex and the unique cut vertex of \(\Delta(G)\). If \(r\) is non-complete, then \(\Delta(G)\) is a graph of diameter \(3\), and it can have at most two cut vertices.
  \end{enumeratei}
\end{thmA}

\begin{proof}
By Theorem~\ref{ultimo},  the vertex set \(\V G\) of \(\Delta(G)\) can be partitioned in two subsets, each inducing a complete subgraph of \(\Delta(G)\): we write the part containing \(r\) as \(\{r\}\cup\alpha\), and we denote by \(\beta\) the other one (note that both \(\alpha\) and \(\beta\) are non-empty in this situation). 

Since the graph  \(\Delta(G)-r\) is not  connected, there are no edges of \(\Delta(G)\) having one extreme in \(\alpha\) and the other extreme in \(\beta\). We conclude that \(\Delta(G)-r\) is a graph whose connected components are the two cliques \(\alpha\) and \(\beta\), so (b) is proved.

On the other hand, since the existence of a cut vertex \(r\) for \(\Delta(G)\) implies that \(\Delta(G)\) is connected by Proposition~\ref{disconnected}, \(r\) must be adjacent to some vertex of \(\beta\), and we have the following dichotomy that proves (c).

\begin{enumeratei}
\item[\(\bullet\)]{\sl The cut vertex \(r\) is a complete vertex}. Then, \(r\) is obviously the unique complete vertex and the unique cut vertex of \(\Delta(G)\). 
\item[\(\bullet\)]{\sl The graph \(\Delta(G)\) has no complete vertices at all}. In this situation, it follows at once that a minimal path connecting a vertex in \(\alpha\) to a vertex (in \(\beta\)) not adjacent to \(r\) has length \(3\). Recalling that, whenever \(\Delta(G)\) is connected, its diameter is at most \(3\) (and a characterization of groups for which the bound is attained can be found in \cite{CD}), the claim of (c) concerning the diameter is proved. Also, if \(t\) is another cut vertex of \(\Delta(G)\), then it is easily seen that \(t\) lies in \(\beta\), and \(\{r,t\}\) is the unique edge of \(\Delta(G)\) involving a vertex in \(\{r\}\cup\alpha\) and a vertex in \(\beta\). As a consequence, \(\Delta(G)\) has at most two cut vertices.
\end{enumeratei}

Finally note that, in both the situations described above, the graph \(\Delta(G)\) has at most one complete vertex. Therefore we can apply Theorem~A of \cite{CDPS12}, which yields conclusion (a) and completes the proof.
\end{proof}

\section{Proof of Theorem B}

We will now tackle the substantial part of our analysis. Before proving Theorem~B in full, we will treat separately one of the cases that may occur (namely, the situation that leads to conclusion \({\bf (I)}\) in the statement of Theorem~B). Recall that, for a group \(G\), we defined $\nu(G)$ as the set of the primes $t\in\pi(G)$ such that $G$ has a normal Sylow $t$-subgroup; also, in the following statement, \(\frat G\) denotes the Frattini subgroup of the group \(G\). 

\begin{theorem}\label{completecut1} Let \(G\) be a reduced group such that \(\Delta(G)\) has a cut vertex \(r\), and let \(R\) be a Sylow \(r\)-subgroup of \(G\). Denoting by \(\alpha\) and \(\beta\) the vertex sets of the two complete connected components of \(\Delta(G)-r\), assume that
  $\nu(G) \cap \alpha = \emptyset = \nu(G) \cap \beta$.  Then, up to interchanging \(\alpha\) and \(\beta\), the following conclusions hold.
\begin{enumeratei}
\item \(\fit G = R\). 
\item Set \(\Phi=\frat R\) and \(K=K_p(G)\), for some \(p\in\alpha\). Then we have $R/\Phi = K\Phi/\Phi \times \zent{G/\Phi}$, and \(K\Phi/\Phi\) is a chief factor of \(G\) whose centralizer in \(G\) is \(R\). 
Furthermore, setting \(\o G=G/R\), we have that \(\fit{\o G}\) is cyclic, it is the \(\alpha\)-Hall subgroup of \(\o G\), and it acts fixed-point freely and irreducibly on $K\Phi/\Phi$. Finally, \(\beta\) consists of a single prime \(q\), \(G\) is \(q\)-nilpotent  and \(|\o G/\fit{\o G}|=q\).

\end{enumeratei}
\end{theorem}

\begin{proof}

  \medskip
  An application of Proposition~\ref{pieces} to the sets $\alpha$ and $\beta$ yields (up to interchanging \(\alpha\) and \(\beta\)) that $K=K_p(G) \nor G$ for all $p \in \alpha$, and $K < K_q(G)$ for all $q \in \beta$. 
  As $\pi(G) = \{ r \}\cup\alpha \cup \beta$, and $K_t(G)$ is a $t'$-subgroup for all
  $t \in \alpha \cup \beta$, we see that $K$ is an $r$-group, so $K \leq R $.

 Let now  \(L\trianglelefteq G\) be such that \(K/L\) is a chief factor of \(G\). An application of Proposition~\ref{gamma}(b) to the factor group \(G/L\) (together with Theorem~2.1 of \cite{MW}) yields that  $\o G = G/\cent G{K/L}$ is a subgroup of the  group of semilinear maps $\Gamma(K/L)$  on \(K/L\), with the cyclic group \(\fit{\o G}\) lying in the subgroup $\Gamma_0(K/L)$ of multiplication maps, and acting (fixed-point freely and) irreducibly on \(K/L\). Also, we get that \(\fit{\o G}\) is the \(\alpha\)-Hall subgroup of \(\o G\) and, taking into account Lemma~\ref{action}, \(\beta\subseteq \pi(\o G/\fit{\o G})\subseteq \beta\cup\{r\}\). As we will see, it turns out that \(\cent G{K/L}\) is in fact \(R\). We proceed through a number of steps.

\smallskip
{\bf{Step 1.}} {\sl The order of \(\o G/\fit{\o G}\) is a power of a prime \(q\) in~\(\beta\)} (hence \(\beta\) consists of a single prime). 

For a proof by contradiction, assume that \(|\o G/\fit{\o G}|\) is divisible by two distinct primes \(q\) and \(t\) (where \(q\in\beta\) and possibly \(t=r\)), let \(\o Q\) be a Sylow \(q\)-subgroup and \(\o T\) a Sylow \(t\)-subgroup of \(\o G\); setting \(|K/L|=r^m\),
 we observe that \(\o T\) is cyclic, hence the order of \(\cent{\Gamma_0(K/L)}{\o T}\) is \(r^{m/|\o T|}-1\) (see \cite[Lemma~3(i)]{D}). Observe also that there exists a primitive prime divisor \(s\) of \(r^{m/|\o T|}-1\): in fact, this is not the case only if \(m/|\o T|=2\) or \(r^{m/{|\o T|}}=2^6\). But in the former situation, by Proposition~\ref{gamma}(c), we have \(q=2\) against the fact that \(q\) does not divide \(r^m-1\); on the other hand, if \(r^{m/{|\o T|}}=2^6\), then \(q=3\) divides \(2^6-1\), again a contradiction. Now, \(s\) is certainly a divisor of \(r^m-1\), but in fact it also divides \((r^m-1)/(r^{m/|\o Q|}-1)\); otherwise, \(s\) is a common divisor of \(r^{m/|\o T|}-1\) and \(r^{m/|\o Q|}-1\), thus it divides \(r^{d}-1\) where $d={\rm{g.c.d.}}(m/|\o Q|,m/|\o T|)$ and (since \(s\) is a primitive prime divisor of \(r^{m/|\o T|}-1\)) we get that \({m/|\o T|}\) divides \({m/|\o Q|}\), a clear contradiction. Again by Proposition~\ref{gamma}(c), it follows that \(s\) divides \(|\fit{\o G}|\), i.e., there exists an element \(\o x\) of \(\fit{\o G}\) whose order is \(s\); recalling that \(\Gamma_0(K/L)\) is cyclic and it has a unique subgroup of order \(s\), we deduce that \(\o x\) is centralized by \(\o T\). Since, as already observed, \(s\) does not divide \(r^{m/|\o Q|}-1=|\cent{\Gamma_0(K/L)}{\o Q}|\), we deduce that \(\o x\) is not centralized by any Sylow \(q\)-subgroup of \(\o G\), whence \(q\) lies in \(\pi_{\o G}(\o x)\). Also, if \(\o y\) is a generator of \(\o T\), certainly \(\o y\) does not centralize \(\fit{\o G}\); as a consequence, \(\pi_{\o G}(\o y)\) contains a prime \(p\) in \(\alpha\). We conclude that \(|(\o {xy})^{\o G}|\) is divisible by \(pq\), which is not the case. This contradiction shows that \(|\o G/\fit{\o G}|\) is a power of \(q\in\beta\), as claimed.

\smallskip
{\bf{Step 2.}} {\sl \(G\) is \(q\)-nilpotent.} 

In fact, if we assume the contrary, then \(G\), and hence \(\o G\), is \(p\)-nilpotent for every \(p\) in \(\alpha\) (see Proposition~\ref{nilpotency}), but this implies that \(\fit{\o G}\) is central in \(\o G\), which is definitely not the case. 

\smallskip
{\bf{Step 3.}} {\sl The order of \(\o G/\fit{\o G}\) is \(q\).} 

For a proof by contradiction, assume \(|\o G/\fit{\o G}|=q^a\) with \(a>1\). Let \(\o Q\) be a Sylow $q$-subgroup of \(\o G\) and consider a subgroup \(\o{Q_0}\) of \(\o Q\) such that \(|\o{Q_0}|=q^{a-1}\). Writing \(m=q^ab\), we have \(|\cent{\Gamma_0(K/L)}{\o {Q_0}}|=r^{bq}-1\), whereas \(|\cent{\Gamma_0(K/L)}{\o {Q}}|=r^{b}-1\). If \(|\cent{\fit{\o G}}{\o {Q}}|\) is strictly smaller than \(|\cent{\fit{\o G}}{\o {Q_0}}|\), then we can choose \(\o x\in \cent{\fit{\o G}}{\o {Q_0}}\) whose conjugacy class size in \(\o G\) is divisible by~\(q\); on the other hand, a generator \(\o y\) of \(\o {Q_0}\) does not centralize \(\fit{\o G}\), hence its conjugacy class in \(\o G\) has a size divisible by a prime \(p\in\alpha\). But now we get the contradiction that \(pq\) divides \(|(\o{xy})|^{\o G}\). In view of this, it will be enough to show that \(|\cent{\fit{\o G}}{\o {Q}}|<|\cent{\fit{\o G}}{\o {Q_0}}|\) holds.

Recalling that \(\Gamma_0(K/L)\) is a cyclic group, what we need to prove is \[{\rm{g.c.d.}}(r^b-1,\;|\fit{\o G}|)\neq{\rm{g.c.d.}} \left((r^b-1)\left(\dfrac{r^{bq}-1}{r^b-1}\right),\;|\fit{\o G}|\right).\] Assuming the contrary, and considering that \((r^{bq}-1)/(r^b-1)\) is a divisor of \(|\fit{\o G}|\) by Proposition~\ref{gamma}(c), we would get that \((r^{bq}-1)/(r^b-1)\) divides \(r^b-1\), hence \(r^{bq}-1\) divides \((r^b-1)^2\). Since it is not difficult to see, as we did above, that \(r^{bq}-1\) has a primitive prime divisor, we reached a contradiction, and our claim is proved.

\smallskip
{\bf{Step 4.}} {\sl \(R\) is a normal subgroup of \(G\).} 

Recalling that every prime of \(\alpha\) is not adjacent to \(q\) in \(\Delta(G)\), Proposition~\ref{nilpotency}(b) yields that there exists an \(\alpha\)-Hall subgroup \(A\) of \(G\), and \(A\) is abelian. Observe also that \(AK\) is a normal subgroup of \(G\), as $KP = [G, P] P \nor G$ for every $P \in \syl pG$ and $p \in \alpha$. Choosing (again) \(L\trianglelefteq G\) such that \(K/L\) is a chief factor of \(G\), for our purposes (and for this step only) we can clearly assume that the \(r\)-subgroup \(L\) is trivial. By Proposition~\ref{gamma}, we know that \(K\) has a complement \(H\) in \(G\), and this \(H\) can be chosen to contain \(A\), so that \(A=AK\cap H\) is a normal subgroup of \(H\). Setting \(A_0=A\cap\cent H K\) we observe that, for every \(p\in\alpha\), we have  \(\oh p {A_0}\trianglelefteq H\) because \(A_0\trianglelefteq H\); but \(\oh p {A_0}\) is clearly normalized by \(K\) as well, so we have \(\oh p {A_0}\leq\oh p G\). Now, Proposition~\ref{gamma}(e) yields that \(\oh p{A_0}\) lies in \(\zent G\) and, as this holds for every choice of \(p\in\alpha\), we deduce that \(A_0\leq\zent G\); in particular, \(A_0\) centralizes a Sylow \(r\)-subgroup \(R_0\) of \(\cent H K\). Recalling that \(H/\cent H K\) is an \(r'\)-group (because \(\o G/\fit{\o G}\) is a \(q\)-group by step 1), we have that \(R_0K\) is a Sylow \(r\)-subgroup of \(G\), and it is enough to show that \(R_0\) is normal in \(\cent H K\) (thus in \(H\)) in order to get \(R_0K\trianglelefteq G\). But the normality of \(R_0\) in \(\cent H K\) follows at once from the fact that \(\cent H K\) is \(q\)-nilpotent, with normal \(q\)-complement \(R_0\times A_0\).

\smallskip
{\bf{Step 5.}} {\sl We have \(\fit G=R\).} 

Let \(U\) be a complement for \(R\) in \(G\): we have to show that \(\fit G\cap U=1\). Setting \(S=\fit G\cap U\), our first remark is that \(S\) lies in \(\zent G\). In fact, \(S\) is certainly normal in \(G\); thus, writing \(S=S_q\times S_{\alpha}\) as a direct product of its Sylow \(q\)-subgroup and its Hall \(\alpha\)-subgroup, we have that both \(S_q\) and \(S_{\alpha}\) are normal in \(G\). But \(G\) is \(q\)-nilpotent with abelian Sylow \(q\)-subgroups, therefore \(S_q\) is central in \(G\). On the other hand, considering the usual normal subgroup \(L\) of \(G\) such that \(K/L\) is a chief factor of \(G\), we have that \(S_{\alpha}L/L\) is a normal subgroup of \(G/L\) intersecting \(K/L\) trivially, so Proposition~\ref{gamma}(d) yields \([S_{\alpha},Q]\leq L\) where \(Q\) is a Sylow \(q\)-subgroup of \(G\). Now, \([S_{\alpha},Q]\leq L\cap S_{\alpha}=1\), thus \(S_{\alpha}\) is centralized by a Sylow \(q\)-subgroup of \(G\). Since \(G\) has abelian Hall \(\alpha\)-subgroups and \(S_{\alpha}\) centralizes \(R\), we conclude that \(S_{\alpha}\) lies in \(\zent G\) as well, so \(S\leq \zent G\). This step can be concluded by observing that no prime divisor of \(|S|\) can divide \(|G'\cap\zent G|\), because \(G\) has abelian Sylow subgroups for each of these primes (see \cite[Theorem~5.3]{I}); as a consequence, \(G'\cap S=1\), and our assumption that \(G\) is a reduced group forces \(S=1\). Thus, we proved claim~(a) of our statement.  

\smallskip
{\bf{Step 6.}} 
The last step is devoted to the proof of claim (b). We start by observing that, for every prime \(p\) in \(\alpha\) and \(P\in\syl p G\), we have \(K=[R,P]\). In fact, we know that \(K=[K,P]\leq[R,P]\); on the other hand, \([R,P]\) is a \(p'\)-subgroup of \([G,P]\), and it is therefore contained in the normal \(p\)-complement \(K\) of \([G,P]\). Taking into account that, as remarked in step 4, an \(\alpha\)-Hall subgroup \(A\) of \(G\) is abelian, we thus get \(K=[R,A]\). Also, an application of \cite[Proposition~3.1]{CDPS13} to the factor group \(G/\Phi\) (recall that here \(\Phi\) is defined as \(\frat R\)) yields that \(K\Phi/\Phi\) is a minimal normal subgroup of \(G/\Phi\), so \(L\) can be chosen to be \(\Phi\cap K\), and \(K\Phi/\Phi\) is isomorphic to \(K/L\) as a \(G\)-module. Now, by Fitting's decomposition we have \(R/\Phi=K\Phi/\Phi\times Z/\Phi\), where \(Z/\Phi\) is set to be \(\cent{R/\Phi}A\) (note that \(Z\) is a normal subgroup of \(G\), as \(AR\nor G\)); but since \(Z/L\) is a normal subgroup of \(G/L\) intersecting \(K/L\) trivially, Proposition~\ref{gamma}(d) yields that \([Z,Q]\leq L\) (where \(Q\) is a Sylow \(q\)-subgroup of \(G\)) and, in particular, \(Q\) centralizes \(Z/\Phi\). We conclude that \(Z/\Phi\) lies in \(\zent{G/\Phi}\) (in fact, equality clearly holds), and \(\cent G{K/L}=\cent G {K\Phi/\Phi}=\cent G{R/\Phi}=R\). Now all the remaining claims in (b) follow by the description of \(\o G=G/R\) that we made in the previous parts of this proof.
\end{proof}

\begin{rem} \label{rem} Assume that \(G\) is a reduced group satisfying the hypothesis of the previous result, so, \(\Delta(G)\) has a cut vertex \(r\) and \(G\) does not have any normal Sylow subgroup except (eventually) for the prime \(r\). We can summarize the conclusions of Theorem~\ref{completecut1}, taking into account the notation introduced therein, as follows. 

Writing \(Z/\frat R\) for the center of \(G/\frat R\), the factor group \(G/Z\) is isomorphic to a subgroup of the affine semilinear group \({\rm A}\Gamma(R/Z)\). Also, the Fitting subgroup of \(G/Z\) is \(R/Z\) and, if \(F/Z\) is the second Fitting subgroup of \(G/Z\), then \(F/R\) is the cyclic Hall \(\alpha\)-subgroup of \(G/R\); as for the top section \(G/F\), it is a group of order~\(q\). Furthermore, we observe that \(\Delta(G/Z)\) is the same as \(\Delta(G)\).

So, the groups appearing as an output in conclusion {\bf(I)} of Theorem~B are well understood (at least as concerns the section over the Frattini subgroup of their normal Sylow \(r\)-subgroup), and they are essentially certain groups of affine semilinear maps.
\end{rem} 

\bigskip
We are now ready to prove Theorem~B, that was stated in the Introduction.

\begin{proof}[Proof of Theorem~B]

 We start by assuming that $G$ is a reduced group whose graph $\Delta(G)$ has a cut vertex $r$ and, as usual, we denote by \(\alpha\) and \(\beta\) the vertex sets of the two complete connected components of \(\Delta(G)-r\). By Theorem~A, we know that $G$ is solvable. 
  Let $A \in \hall{\alpha}G$, $B \in \hall{\beta}G$ and $R \in \syl rG$ be such that $ AB$ and $AR$ are subgroups of $G$.
  Since no vertex of $\alpha$ is adjacent in $\Delta(G)$ to any vertex of $\beta$,  Proposition~\ref{nilpotency}
  yields that both $A$ and $B$ are abelian. 
  
 \smallskip 
 Recalling that $\nu(G)$ is the set of the prime divisors $t$ of $|G|$ such that $G$ has a normal Sylow $t$-subgroup, let us first assume that $\nu(G) \cap (\alpha \cup \beta) = \emptyset$. Our aim is to show that conclusion \({\bf (I)}\) holds in this case. By Theorem~\ref{completecut1}, we have that \(R=\fit G\); moreover, \(H = AB\) has a Sylow \(q\)-subgroup $B$ of order \(q\), where $\{ q \} = \beta$, and a cyclic normal \(q\)-complement \(A=\fit H\). 

  Now, assume that \(x\in R\) does not centralize any conjugate  \(A^y\) with  \( y \in R\) (hence, any \(G\)-conjugate of \(A\) at all). As a consequence, there exists a prime \(p\in\pi(A)\) which divides the size of \(x^G\). Since, as remarked above, \(p\) is not adjacent to \(q\), certainly \(x\) is centralized by a Sylow \(q\)-subgroup of \(G\); moreover, if there exists an element \(w\) in \(\cent A x\setminus\zent H\), then some prime in \(\pi(H)\) has to divide \(|w^H|\), and this prime is certainly \(q\) because \(A\) is abelian. But now \(q\) divides \(|w^G|\) as well (because \(H\) is isomorphic to \(G/R)\), so \(pq\) divides \(|(xw)^G|\), a contradiction. We deduce that \(\cent A x\) lies in \(\zent H\), and we get case {\bf(I)}.

  \smallskip
  Assume now, by the  symmetry of $\alpha$ and $\beta$,  that there exists a prime $ t \in \nu(G) \cap \alpha$.
  Note that this implies, by Proposition~\ref{nilpotency}(a), that $G$ is $q$-nilpotent for all $q \in \beta$. Hence, the $\beta$-complement $AR$ is a normal subgroup of $G$, and $G' \leq AR$. 

  Observe first that $\nu(G) \cap \beta = \emptyset$. In fact, if $q \in \nu(G) \cap \beta$, then
  $G = \cent GT \cup \cent GQ$, where $T \in \syl tG$ and $Q \in \syl qG$, which is not possible.
  Next, we claim that \(\alpha\subseteq\nu(G)\). In fact assume, working by contradiction, that \(\pi=\alpha\setminus\nu(G)\) is non-empty; then, as shown in step 5 of the proof of \cite[Theorem~A]{DPSS}, we have \(K_q(G)<K_p(G)\) for all \(q\in\beta\) and \(p\in\pi\).  
  Also, Proposition~\ref{pieces} yields that there exists \(K\nor G\) such that $K_q(G) = K$ for all $q \in \beta$.
  In particular, this implies that $\pi(K) \subseteq \nu(G) \cup \{ r \}$. 
  Let now $L \leq K$ be a normal subgroup of $G$ such that $K/L$ is a chief factor of $G$ and let
  $\o G = G/ \cent G{K/L}$. Observe that, as the Fitting subgroup of $G$ centralizes every chief
  factor of $G$,  the group $\o G$ is a $\nu(G)'$-group. 
  By Proposition~\ref{gamma}(b), for all $p \in \pi$ the Sylow $p$-subgroup $\o P$ of $\o G$ intersects
   $\fit {\o G}$ trivially, and $\o{B}$ acts fixed-point freely on $K/L$. As $\o B$ is central in $\o G$ (because \(G\) is \(q\)-nilpotent for every \(q\) in \(\beta\)) and $\o A$ is
  abelian, it follows that $K_p(\o{G}) = [\o G, \o P]$ is an $r$-group. Hence, $r$ does not divide $|K/L|$, so
    $K/L$ is a $t$-group for some $t \in \nu(G)$. For any non-trivial $ xL \in K/L$, we have
    $\pi(\o B ) \subseteq \pi_G(x)$, so $x$ is centralized by a Sylow $p$-subgroup $P_0$
    of $G$. Since $P_0$ is not contained in $\cent G{K/L}$, there exists $ y \in  P_0$ such that $t \in \pi_G(y)$, and hence 
    $\pi_{G}(xy)$ contains both  $\pi(\o B)$ and $t$, a contradiction. 

   Hence, $\alpha \sbs \nu(G)$ and  $A$ is a normal subgroup of $G$. 
   We will show, next, that either $R$ or $AB$ is a normal subgroup of $G$.
   We first observe that, for every $q \in \beta$, there exists $Q \in \syl qG$ such that  $RQ$ is a subgroup of $G$. As $G_0 = PRQ$ is isomorphic to
   a normal section of $G$, the graph $\Delta(G_0)$ is a subgraph of $\Delta(G)$ so, in particular, \(\{p,q\}\) is not an edge of \(\Delta(G_0)\).
   As both $P$ and  $PR$ are normal subgroups of $G_0$,  by Lemma~\ref{three}, $R$ commutes with either $P$ or $Q$; in the first case $R$ is normal in $G_0$ and in the second case $PQ$ is normal in $G_0$.
   Thus the subgroup $AB$ is non-abelian; otherwise,  either $P$ or $Q$ would be central in $G$, a contradiction.
   So, by a suitable choice of $p \in \alpha$ and  $q \in \beta$, we can assume that $[P, Q] \neq 1$. Since $PQ$ is either a normal subgroup of $G_0$ or isomorphic to a quotient of $G_0$, there are elements $x \in P$ and $y \in Q$ such that $q \in \pi_{G_0}(x) \subseteq \pi_G(x)$ and
   $p \in \pi_{G_0}(y) \subseteq \pi_G(y)$.  
   Assume first that $[R,P] = 1$ (so $R \nor G_0$) and let $t \in \alpha$ and $T \in \syl tG$.
   If $[R, T] \neq 1$, we consider  $w \in R$, $w \not\in \cent RT$ and get
   $\{t, q\} \subseteq \pi_G(xw)$, a contradiction. So, in this case, $R$ commutes with $A$ and
   hence $R$ is a normal subgroup of $G$.

   Assume, on the other hand, $[R, Q] = 1$.
   Let $\o G = G/A$. If $[\o R, \o B] \neq 1$, then there is $w \in R$ and $t \in \beta$ such that $t \in \pi_{\o G}(\o w)$.
   Thus, $\{t, p\} \subseteq \pi_G(yw)$, a contradiction. Therefore, in this case,
   $G/A \simeq R \times B$ and $AB$ is the normal $r$-complement of $G$. 

   \smallskip
   We now suppose that $R$ is normal in $G$. Hence, $AR = A \times R = \fit G$, because
   $B \cap \fit G \leq \zent G$ has trivial intersection with $G'$ and $G$ is reduced.
Let $Z = \zent {AB}$ and  note that $Z \cap A=\cent A B$ is (by Fitting's decomposition) a central direct factor of $G$, so $Z = \oh{\beta}{AB} \leq B$  as $G$ is reduced. Note that $Z < B$, as otherwise $A$ would be central in $G$. 
   
Let $b \in B \setminus Z$ and $a \in \cent Ab$.
If $a \neq 1$, then there exists $q \in \beta$ such that $q \in \pi_G(a)$.
Also, there is $p \in \alpha$ such that $p \in \pi_G(b)$, so we get the contradiction
$\{p, q \} \sbs \pi_G(ab)$.
Hence $AB/Z$ is a Frobenius group, with kernel $AZ/Z$.

If $[R, B] = 1$, then  $G = R \times AB$ and $R$ is non-abelian; so we are in case {\bf(IIb(i))}.
 
If $R$ is abelian, then $\cent RB$ is a central direct factor of $G$ and hence $\cent RB = 1$ as
$G$ is reduced. So, for every non-trivial $x\in R$ we have $\pi_G(x) \cap \beta \neq \emptyset$ and
hence $\cent Bx \leq \cent BA = Z$ by Lemma~\ref{product}. Note also that in this case $Z \neq 1$,
as otherwise the graph $\Delta(G)$ would be disconnected by Proposition~\ref{disconnected}.
Finally, as $\cent BR \leq Z$ we see that $\cent BR \leq \zent G$. Since $B \cap G' = 1$ and $G$
is reduced, we see that $\cent BR = 1$.
Thus, we have case {\bf(IIa)}. 

Assume now that $R$ is non-abelian and that $[R, B] \neq 1$.
Consider  an element $x \in R$ such that $\cent Gx  R < G$, i.e. such that $\cent Gx$ does
not contain any conjugate of $B$ in $G$. Then there exists a prime $q \in \beta$ such that
$q \in \pi_G(x)$ and again Lemma~\ref{product} implies that $\cent Bx \leq Z$. So, we have case {\bf(IIb(ii))}.

\smallskip
For the last case, assume  that $AB$ is the normal $r$-complement of $G$.
By the Frattini argument we can choose $R \leq \norm GB$; therefore, $BR$ is a subgroup of $G$ and, since \(AR\nor G\), we have \(R=AR\cap RB\nor BR\). As a consequence, \(B\) and \(R\) are direct factors of \(BR\) (i.e., \(BR\) is nilpotent).  
Let $R_0 = \oh rG$, and observe
that we can assume that $R_0 < R$, as otherwise $G = R \times AB$ and we are again in case {\bf(IIb(i))}.
As above we observe that, as $G$ is reduced, we have $\fit G = A \times R_0$. So, $R_0 = \cent {BR}A$.
Write $A = A_0 \times C$, where $A_0 = [A,B]$ and  $C = \cent AB$. 
We show that $A_0BR/R_0$ is a Frobenius group. In fact, if $x \in A_0\setminus\{1\}$,
then $q \in \pi_G(x)$ for some $q \in \beta$, and hence Lemma~\ref{product} implies that
$\cent {BR}x \leq \cent{BR}A = R_0$.
Moreover, if $R$ is abelian, then
$R_0 = 1$ as $G$ is reduced. Hence, if $C = 1$, then $\Delta(G)$ would be disconnected by Proposition~\ref{disconnected}, against our assumptions, and we reached conclusion {\bf(III)}.

\bigskip
We now start proving the ``if part" of Theorem~B.
We recall that, if $\pi_1, \pi_2, \ldots, \pi_n$ are disjoint sets of primes and $g$ is an element of $G$,
one can uniquely write $g = g_{\pi_1}g_{\pi_2}\cdots g_{\pi_n}$, where each $g_{\pi_i}$ is a
$\pi_i$-element and a power of $g$; we call this the \emph{standard decomposition} of $g$ (with respect to $\pi_1, \pi_2, \ldots, \pi_n$). Note that then $\cent Gg = \bigcap_{i = 1}^n \cent G{g_{\pi_i}}$.

\medskip
Let us assume {\bf(I)}: in this case $B = Q$ is a Sylow \(q\)-subgroup of $G$.  
We first show that \(q\) is not adjacent in \(\Delta(G)\) to any prime in $\alpha$. 

What we have to prove is that, for a fixed \(p\in\pi(A)\) and \(g\in G\), the size of \(g^G\) is not divisible by \(pq\).
We have the standard decomposition   \(g=g_rg_{\alpha}g_q\),  where we can assume, up to conjugation in \(G\), that  \(g_r\in R\), \(g_{\alpha}\in A\) and \(g_q\in Q_0\), for some $Q_0 \in \syl qG$. If \(g_q\neq 1\), then \(\langle g_q\rangle=Q_0\) (recall that \(|Q_0|=q\)) centralizes \(g\), therefore \(q\nmid |g^G|\). To the end of showing that \(|g^G|\) is not divisible by \(pq\) we will therefore assume \(g_q=1\). 

Let us consider the case when \(g_r\) is centralized by a conjugate \(A^v\) of \(A\), with \(v\in R\). Since \(g_{\alpha}\) is a \(\pi(A)\)-element of \(\cent G {g_r}\) and \(A^v\) is a Hall \(\pi(A)\)-subgroup of \(\cent G {g_r}\), there exists \(c\in \cent G {g_r}\) such that \(g_{\alpha}\) lies in \(A^{vc}\). But \(A^{vc}\) is abelian, so \(g_{\alpha}\) is centralized by \(A^{vc}\), as well as \(g_r\). The conclusion is that \(g=g_rg_{\alpha}\) is centralized by the Hall \(\pi(A)\)-subgroup \(A^{vc}\) of \(G\), whence \(p\nmid |g^G|\) and we are done in this case.

The last situation that has to be considered is when \(g_r\) is not centralized by \(A^v\) for any \(v\in R\). Set $H = AB$.  Then, by our assumptions, a \(G\)-conjugate \(Q^u\) of \(Q\) lies in \(\cent G {g_r}\), and \(\cent A {g_r}\leq\zent H\); in particular, we get \(g_{\alpha}\in\zent H\), thus \(o(g_{\alpha})\mid|\zent H|\). Choose now an \(r\)-complement \(H_1\) of \(\cent G {g_r}\) which contains \(Q^u\), and let \(A_1\) be the (cyclic) \(\pi(A)\)-Hall subgroup of \(H_1\). Since \(g_{\alpha}\) is a \(\pi(A)\)-element of \(\cent G {g_r}\), there exists \(c\in\cent G {g_r}\) such that \(g_{\alpha}\) lies \(A_1^c\). Observe that \(o(g_{\alpha})\) divides the order of \(\zent {H_1^c}\leq A_1^c\) and, \(A_1^c\) being cyclic, its unique subgroup of order \(o(g_{\alpha})\) (i.e., \(\langle g_{\alpha}\rangle\)) is forced to lie in \(\zent{H_1^c}\). We conclude that \(g_{\alpha}\) lies in \(\zent{H_1^c}\), and therefore \(g_{\alpha}\) is centralized by \(Q^{uc}\). But \(Q^u\) lies in \(\cent G {g_r}\), so the same holds for \(Q^{uc}\) (recall that \(c\in \cent G {g_r}\)) and \(Q^{uc}\) centralizes \(g_r\) as well. As a consequence, in this situation the size of the conjugacy class of \(g=g_rg_{\alpha}\) in \(G\) is not divisible by \(q\). 

So we finished the proof that \(q\) is not adjacent in \(\Delta(G)\) to any prime in \(\alpha\), which also implies (by Proposition~\ref{nilpotency}(b)) that the vertices in \(\alpha\) are pairwise adjacent in \(\Delta(G)\).

Finally, we observe that $r$ is a complete vertex of $\Delta(G)$. In fact, assuming the contrary, our graph would have no complete vertices, and therefore \(G\) would be metabelian by Theorem~C of \cite{CDPS12}. But this is not the case, as \(G\) has Fitting height~\(3\). We conclude that \(r\) is a cut vertex of \(\Delta(G)\) and we are done. 
 
 \smallskip
Let us assume now case {\bf(II)}: $\fit G = A \times R$, $Z = \zent{AB} <B$, and $AB/Z$ is a Frobenius group with kernel \(AZ/Z\) (note that \(Z= \cent BA\) and \(AZ/Z\simeq A\)).  

\smallskip
{\bf(IIa)} ($R$ is abelian, $\cent BR = 1$ and $\cent Bx \leq Z \neq 1$ for every non-trivial $x \in R$).
Note that, as $G$ is reduced, the vertex set of $\Delta(G)$ is $\alpha \cup \beta \cup \{ r \}$. 
We first show that, for $p \in \alpha$ and $q \in \beta$, $p$ and $q$ are non-adjacent in $\Delta(G)$. In fact, 
let $g \in G$ and consider the standard decomposition $g = g_{\alpha}g_r g_{\beta}$, with $g_{\alpha} \in A$, $g_r \in R$ and, up to
conjugation, $g_{\beta} \in B$. Assuming that \(pq\) divides \(|g^G|\), we clearly have \(p\in\pi_G(g_{\beta})\), which implies \(g_{\beta}\not\in Z\). Since \(AB/Z\) is a Frobenius group with kernel \(AZ/Z\), and \(g_{\beta}\) commutes with \(g_{\alpha}\), we deduce that \(g_{\alpha}\) must be trivial and so \(g_r\neq 1\) (otherwise \(g=g_{\beta}\) would not lie in a conjugacy class having size divisible by \(q\)). But now we get \(g_{\beta}\in\cent B{g_r}\leq Z\), a contradiction.
As in case {\bf(I)}, this also implies that both \(\alpha\) and \(\beta\) induce complete subgraphs of \(\Delta(G)\).
Finally, we observe that $\Delta(G)$ is connected by Proposition~\ref{disconnected}, so $r$ is a cut vertex of $G$, as wanted. 

Note also that, as easily seen, every element in \(B\setminus Z\) has a $G$-conjugacy class size divisible by \(r\) and by all the primes in \(\alpha\), therefore \(\alpha\cup\{r\}\) induces a complete subgraph of \(\Delta(G)\).

\smallskip
{\bf(IIb(i))} ($G = R \times AB$). 
In this case, it is clear that $\Delta(G)$ is the join of a graph with one vertex $r$ and a
disconnected graph with connected components of vertex sets $\alpha$ and $\beta$.

\smallskip
{\bf(IIb(ii))} ($R$ is non-abelian, and $\cent Bx \leq Z$ for all $x \in R$ with $\cent Gx R <G$).
Let $g \in G$, and write \(g\) in its standard decomposition as $g_{\alpha}g_r g_{\beta}$, with $g_{\alpha} \in A$, $g_r \in R$ and, up to
conjugation, $g_{\beta} \in B$. Assume, working by contradiction, that  $\{p, q\} \subseteq \pi_G(g)$
for some $p \in \alpha$ and $q  \in \beta$;
then $p \in \pi_G(g_{\beta})$. Thus we have $g_{\beta} \not\in Z$, and hence $g_{\alpha} = 1$, because \(g_{\alpha}\) commutes with \(g_{\beta}\) and $AB/Z$ is a Frobenius group with kernel \(AZ/Z\). But also \(g_r\) commutes with \(g_{\beta}\), therefore, by our assumptions, we have \(\cent G{g_r}R=G\); in particular, there exists a Hall \(\beta\)-subgroup \(B_0\) of \(G\) lying in \(\cent G{g_r}\). Now, \(g_{\beta}\) is a \(\beta\)-element of \(G\) contained in \(\cent G{g_r}\), and so there exists \(c\in\cent G{g_r}\) such that \(g_{\beta}\) lies in \(B_0^c\) (which is abelian). As a consequence, \(B_0^c\) centralizes \(g=g_rg_{\beta}\), and in particular \(q\not\in\pi_G(g)\), contradicting our assumptions. As in case {\bf(I)}, \(G\) being not metabelian, \(r\) is a complete vertex of \(\Delta(G)\) and it is therefore a cut vertex of \(\Delta(G)\), as wanted.

\smallskip
Let us assume the last case {\bf(III)}: \(BR\) is a nilpotent subgroup of \(G\); also, $\fit G = A \times R_0$, with $R_0 < R$, $\cent AR = 1$, 
and $[A,B]BR/R_0$ is a Frobenius group with kernel \([A,B]R_0/R_0\). In the case when $R$ is abelian, in addition we have $R_0 = 1$ and $C = \cent AB \neq 1$.

As before, let $g \in G$ and consider the standard decomposition $g = g_{\alpha} g_{\{ r \} \cup \beta}$, with $g_{\alpha} \in A$,  and, up to
conjugation,  $g_{\{ r \} \cup \beta} \in BR$. Assume, working by contradiction, that  $\{p, q\} \subseteq \pi_G(g)$,
for some $p \in \alpha$ and $q  \in \beta$.
Then $p \in \pi_G(g_{\{ r \} \cup \beta})$.
As $A = [A,B] \times C$, write  also $g_{\alpha} = g_0  g_1$ with $g_0 \in [A, B]$ and $g_1 \in C$, and note that \(g_{\{ r \} \cup \beta}\) centralizes both \(g_0\) and \(g_1\), because \([A,B]\) and \(C\) are normal subgroups of \(G\). 
Since $g_{\{ r \} \cup \beta} \not\in R_0 = \cent{BR}A$, our assumptions imply that $g_0 = 1$, so
$g_{\alpha} \in C$ and hence \(q\not\in\pi_G(g)\), a contradiction. As usual, what we proved implies also that both \(\alpha\) and \(\beta\) induce complete subgraphs of \(\Delta(G)\). Finally we observe that, by Proposition~\ref{disconnected}, $\Delta(G)$ is connected both when
$R$ is non-abelian (in which case \(r\), as in {\bf (I)}, is a complete vertex of \(\Delta(G)\)) and when $R$ is abelian; in fact, in the latter case, we get that $\zent G = 1$ and $G$ is not a Frobenius group. Thus $r$ is a cut vertex of $\Delta(G)$, and the proof is complete.

We also note that every non-trivial element in \([A,B]\) has a $G$-conjugacy class size divisible by \(r\) and by all the primes in \(\beta\), therefore \(\beta\cup\{r\}\) induces a complete subgraph of \(\Delta(G)\). 
\end{proof} 

\section{Discussion of the cases of Theorem~B, and proof of Theorem~C}

Next, we take time for a closer look at the groups that appear in Theorem~B, also deriving some more detailed information about the associated graphs. As a consequence of this discussion, we will get Theorem~C (see Remark~\ref{thmC}). We will also determine, in Example~\ref{ex}, which \(1\)-connected graphs can occur as \(\Delta(G)\) for a finite group \(G\).

So, let \(G\) be a reduced group such that \(\Delta(G)\) has a cut vertex \(r\). As in Theorem~B, we denote by \(\alpha\) and \(\beta\) the vertex sets of the two connected components of the graph \(\Delta(G)-r\) (the description being given up to interchanging \(\alpha\) and \(\beta\)). 

First of all we stress that, in this setting, the groups as in {\bf(I)} are characterized by the fact that they have a normal Sylow subgroup only for the prime \(r\). The structure of these groups has been already discussed in Remark~\ref{rem}, and we do not comment further on that.

As regards the groups in classes {\bf (II)} and {\bf(III)}, they share the property of having normal Sylow subgroups for all the primes in \(\alpha\), whereas the Sylow subgroups for the primes in \(\beta\) are all non-normal. If, in this situation, the group has an abelian normal Sylow \(r\)-subgroup, then it lies in {\bf(IIa)}; if it has a non-abelian normal Sylow \(r\)-subgroup, then we are in case {\bf (IIb)}. On the other hand, if the group does not have a normal Sylow \(r\)-subgroup, then it belongs to class {\bf(III)}. 

Some more remarks:

\medskip
${\bullet}$ 
For a group \(G\) as in {\bf(IIa)}, the cut vertex \(r\) need not be a complete vertex of \(\Delta(G)\). If it is not, as observed in Theorem~A, the graph \(\Delta(G)\) has diameter \(3\). 

More specifically, \(r\) is adjacent to all the primes in \(\alpha\), but it can be non-adjacent to some prime in \(\beta\): in order to have a better understanding of \(\Delta(G)\) in this case, we characterize next the set \(\beta^*\subseteq\beta\) of the vertices of our graph that are non-adjacent to \(r\).

Let \(R\) be the Sylow \(r\)-subgroup of \(G\) and, for $q \in \beta$, let $Q$ be in $\syl qB$. We claim that $q$ lies in $\beta^*$ if and only if
$Q \leq Z= \cent BA$ and $B$ acts fixed point-freely on $[R,Q]$.
In fact, if $Q \not\leq Z$, then $q \in \pi_G(x)$ for some element $x\in A$.
Consider a non-trivial element $y \in Z$ (recall that  $Z \neq 1$); then $r \in \pi_G(y) $ and hence $\{r, q\} \sbs \pi_G(xy)$.
If, on the other hand, there exist non-trivial and commuting elements $x \in [R, Q]$ and $y \in B$, then $\pi_G(xy) \supseteq \pi_G(x) \cup \pi_G(y) \supseteq \{r,q\}$ (recall that $\cent BR = 1$).

Conversely, let $g = g_{\alpha}g_rg_{\beta}$ be the standard decomposition of \(g\), where we can assume, up to conjugation,  $g_{\alpha} \in A$, $g_r \in R$ and $g_{\beta} \in B$. Assume that
$\{r, q\} \sbs \pi_G(g)$ and that $B$ acts fixed point-freely on $[R,Q]$. As $r \in \pi_G(g)$, then $g_{\beta} \neq 1$; so, using the Fitting decomposition of  the abelian group $R$ with respect to the action of $Q$, we get $g_r \in \cent RQ$. Thus $q \in \pi_G(g)$ implies $Q \not\leq \cent B{g_{\alpha}}$, and hence
$Q \not\leq Z$.

\medskip
${\bullet}$ For the groups in {\bf (IIb)} (as well as for those as in {\bf (I)}), the cut vertex $r$ is a complete vertex of $\Delta(G)$.

\medskip
$\bullet$ Finally, let \(G\) be as in {\bf(III)}. Then the cut vertex \(r\) is always adjacent in \(\Delta(G)\) to all the vertices in \(\beta\), and it is a complete vertex if a Sylow \(r\)-subgroup \(R\) of \(G\) is non-abelian. On the other hand, if $R$ is abelian, \(r\) can be non-adjacent to some prime in \(\alpha\) (and, if this happens, then \(\Delta(G)\) has diameter \(3\)): as we did for class {\bf(IIa)}, we characterize next the set \(\alpha^*\subseteq\alpha\) of the vertices of \(\Delta(G)\) that are non-adjacent to \(r\) in this case.


 For $p  \in \alpha$ and  $P \in \syl pA$,  we  show that $p \in \alpha^*$ if and only if $P \leq C=\cent AB$ and $\cent Rx \leq \cent RP$ for all non-trivial $x \in A$.
 In fact, if $P \not\leq C$, then there exists $y \in B$ such that $p \in \pi_G(y)$. Considering a non-trivial $x \in C$, we have $r \in \pi_G(x)$ (as $\cent AR = 1$) and
 hence $\{p, r \} \sbs \pi_G(xy)$. If, on the other hand, there exist non-trivial elements $y \in R \setminus \cent RP$ and $x \in \cent Ay$, then again $\{p, r \} \sbs \pi_G(xy)$.

Conversely, let $g = g_{\alpha}g_rg_{\beta}$ be the standard decomposition of \(g\), where we can assume, up to conjugation,  $g_{\alpha} \in A$, $g_r \in R$ and $g_{\beta} \in B$. Assume that
$\{p, r \} \sbs \pi_G(g)$ and that $\cent Rx \leq \cent RP$ for all non-trivial $x \in A$. As $R$ does not centralize $g$, then $g_{\alpha} \neq 1$ and hence $g_r \in \cent RP$.
So $g_{\beta} \not\in \cent GP$ and hence $P \not\leq C$.

\begin{rem}\label{thmC}
Observe that Theorem~C is an immediate consequence of (Theorem~B and) the analysis carried out above. In fact, the reduced groups whose related graph has two cut vertices are easily seen to be those lying in class {\bf(IIa)} such that the second cut vertex \(t\) is the unique element in \(\beta\setminus\beta^*\) (or, equivalently, the groups lying in class {\bf (III)} such that the second cut vertex \(t\) is the unique element in \(\alpha\setminus\alpha^*\)).
\end{rem}

We close this section by showing that \emph{every $1$-connected graph which is covered by two complete subgraphs does in fact occur as the graph \(\Delta(G)\) for a suitable group~$G$.} (Conversely, every graph of the kind \(\Delta(G)\) which has a cut vertex is \(1\)-connected, as observed in the Introduction, and it is covered by two complete subgraphs by Theorem~\ref{ultimo}.)

\begin{example}\label{ex}
Let $n,  m_1$ be positive integers and $m_0$ a non-negative integer.
Let $b_0 = q_1q_2\cdots q_{m_0}$ and
$b_1 = t_1t_2\cdots t_{m_1}$ where the $q_i$ and the $t_j$ are distinct primes
(meaning also $q_i \neq t_j$, for all $i, j$).
Let $r$, $p_1, p_2, \ldots, p_n$ be distinct primes such that
$r \equiv 1 \pmod{b_0b_1}$ and $p_i \equiv 1 \pmod{b_1}$
for all  $1 \leq i \leq n$; note that they exist by Dirichlet's Theorem on primes in an arithmetic progression.

Let $B_0$ and $B_1$ be cyclic groups of order $b_0$ and $b_1$, 
and $R$ and $A$ cyclic groups of order $r$ and $p_1p_2\cdots p_n$, respectively.
Consider the semidirect product $G = (A \times R) \rtimes (B_0 \times B_1)$
with respect to a Frobenius action of $B_0 \times B_1$ on $R$ and
of $B_1$ on $A$, while $B_0$ acts trivially on $A$.
Then it is easily seen that the graph $\Delta(G)$ is covered by two complete subgraphs (on the sets $\{r, p_1, \ldots, p_n\}$ and $\{q_1, \ldots, q_{m_0}, t_1,\ldots , t_{m_1}\}$), and that $r$ is a cut vertex of $\Delta(G)$ which is
adjacent exactly to the primes $\{ t_1,\ldots , t_{m_1}\}$ (see Figure~1).

\begin{figure}[h]
\label{example1}
   \centering
       {\includegraphics[width=.4\textwidth]{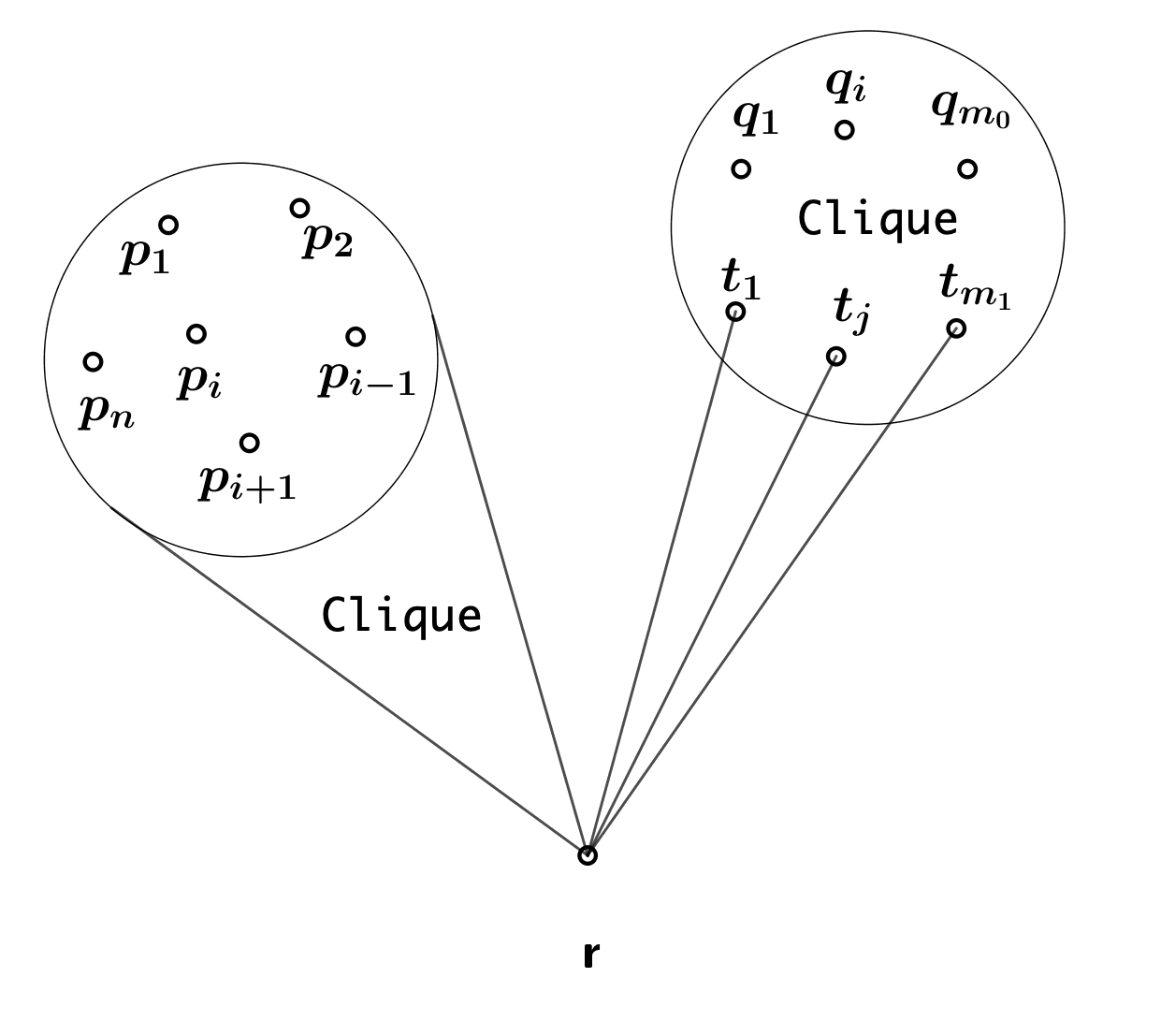}}
       \caption{Example 5.2}
\end{figure}

Observe that $r$ is complete if and only if $G = (A \times R) \rtimes  B_1$, and that there
are two cut vertices if and only if $m_1 = 1$. \end{example}



\section*{Acknowledgements}
This research has been carried out during a visit of the second and fourth authors at the Dipartimento di Matematica e Informatica ``Ulisse Dini" (DIMAI) of Universit\`a degli Studi  di Firenze.  They wish to thank the DIMAI for the hospitality. 



\begin{thebibliography}{99}

\bibitem{BDIP} D. Bubboloni, S. Dolfi, M.A. Iranmanesh, C.E. Praeger,
 \emph{On bipartite divisor graphs for group conjugacy class sizes}, J.  Pure  Appl. Algebra  213  (2009), 1722--1734. 
\bibitem{CD} C. Casolo, S. Dolfi, \emph{The diameter of a conjugacy class graph of finite groups}, Bull. London Math. Soc. 28 (1996), 141--148.

\bibitem{CD2} C. Casolo, S. Dolfi, \emph{Products of primes in conjugacy class sizes and irreducible character degrees}, Israel J. Math. 174 (2009), 403--418. 

\bibitem{CDPS12} C. Casolo, S. Dolfi, E. Pacifici. L. Sanus, \emph{Groups whose prime graph on conjugacy class sizes has few complete vertices }, J. Algebra 364 (2012), 1--12.

\bibitem{CDPS13} C. Casolo, S. Dolfi, E. Pacifici, L. Sanus, \emph{Incomplete vertices in the prime graph on conjugacy class sizes of finite groups}, J. Algebra  376 (2013),  46-57.


\bibitem{D0} S. Dolfi, \emph{Arithmetical conditions on the length of the conjugacy classes of a finite group}, J. Algebra 174 (1995), 753--771.
  
\bibitem{D} S. Dolfi, \emph{On independent sets in the class graph of a finite group}, J. Algebra 303 (2006), 216--224.

\bibitem{DPSS} S. Dolfi, E. Pacifici, L. Sanus, V. Sotomayor, \emph{The prime graph on class sizes of a finite group has a bipartite complement}, preprint 2019, submitted.
 \bibitem{I} I.M. Isaacs, \emph{Finite Group Theory}, American Mathematical Soc.,  2008.
\bibitem{LM}  M.L.  Lewis, Q.  Meng, \emph{Solvable groups whose prime divisor character degree graphs are $1$-connected}, Monatsh Math (2019). doi.org/10.1007/s00605-019-01276-8

\bibitem{MW}  O. Manz, T.R. Wolf, \emph{Representations of Solvable Groups},
Cambridge University Press, Cambridge, 1993.
\end{thebibliography}
\end{document}